\documentclass[11pt,  oneside,  a4paper]{amsart}
\usepackage{amssymb}

\usepackage{amsmath}

\usepackage{amscd,amsopn}

\usepackage{graphicx}
\usepackage{amsmath,  amsfonts,  amssymb,  amsthm}
\usepackage{xypic}
\usepackage{verbatim}

\usepackage{mathrsfs}

\usepackage{stmaryrd}

\usepackage{enumitem}

\usepackage[all]{xy}

\usepackage{textcomp}

\usepackage{amsbsy}

\usepackage{xcolor}

\usepackage{etex}
\usepackage{hyperref}
\hypersetup{
  colorlinks = true,
  linkcolor  = black
}

\usepackage[alphabetic,backrefs,msc-links]{amsrefs}

\usepackage{tikz}
\usepackage{caption}
\usetikzlibrary{decorations.markings}

\DeclareMathOperator{\II}{I}

\DeclareMathOperator{\rank}{rank}
\DeclareMathOperator{\SHI}{SHI}
\DeclareMathOperator{\KHI}{KHI}

\newtheorem{THE}{Theorem}[section]

\newtheorem{thm-defn}[THE]{Theorem/Definition}
\newtheorem{LE}[THE]{Lemma}
\newtheorem{PR}[THE]{Proposition}

\theoremstyle{definition}
\newtheorem{DEF}[THE]{Definition}

\newtheorem{eg}[THE]{Example}

\theoremstyle{remark}
\newtheorem*{rmk}{Remark}

\numberwithin{equation}{section}

\author{Yi Xie}
\title{Earrings, sutures and pointed links}
\date{}

\newcommand{\Address}{{
  \bigskip
  \footnotesize
   Yi Xie, \textsc{Simons Center for Geometry and Physics, State University of New York,
  Stony Brook, NY 11794}\par\nopagebreak
  \textit{E-mail address}: \texttt{yxie@scgp.stonybrook.edu}
}}

\begin{document}
\maketitle
\begin{abstract}
We prove a rank inequality on the instanton knot homology and the Khovanov homology of a link in $S^3$.
The key step of the proof is to construct
 a spectral sequence relating 
Baldwin-Levine-Sarkar's pointed Khovanov homology to a singular instanton invariant for pointed links.
\end{abstract}

\section{Introduction}
The relationship between Khovanov homology and different Floer  theories has been studied a lot since 
Ozsv\'{a}th and Szab\'{o}  discovered the first spectral sequence relating Khovanov homology to Floer homology \cite{OS-ss}.
In this paper we continue studying this topic using instanton Floer homology.

Given a link $L$ in $S^3$ with $m$ components, its Khovanov homology $Kh(L)$ is equipped with a 
\begin{equation*}
R_m=\mathbb{Z}[X_1,X_2,\cdots, X_m]/(X_1^2, X_2^2,\cdots,X_m^2)
\end{equation*}
module structure \cite{Kh-module,HN-module}. Let $\mathbf{X}=(X_1,\cdots,X_m)\in R_m^{\oplus m}$ and $K(\mathbf{X},Kh(L))$ 
be the Koszul complex. 
We prove the following.
\begin{THE}\label{KHI<Kh}
Suppose $L$ is a link in $S^3$, then we have
\begin{equation*}
2\dim_{\mathbb{C}} \KHI(L) \le \rank_{\mathbb{Z}} H(K(\mathbf{X},{Kh}(L)))
\end{equation*}
\end{THE}
Here $\KHI(L)$ is the instanton knot Floer homology defined in \cite{KM:suture,KM-Alex}. The Koszul complex is defined
as the tensor product 
$$Kh(L)\otimes_{R_m}K(\mathbf{X})$$ 
where $K(\mathbf{X})$ is the chain complex
\begin{equation*}
\xymatrix{
  K(\mathbf{X}):=0 \to  R_m \ar[r]^-{\wedge \mathbf{X}} & \Lambda^1 R_m^{\oplus m} \ar[r]^-{\wedge \mathbf{X}} & \Lambda^2 R_m^{\oplus m}
  \ar[r]^-{\wedge \mathbf{X}} & \cdots \ar[r]^-{\wedge \mathbf{X}} & \Lambda^m R_m^{\oplus m}
  \to 0
}
\end{equation*}

\begin{eg}
If $L=U_m$ is an unlink with $m$ components, then $Kh(U_m)\cong R_m$. It is easy to see 
$ H(K(\mathbf{X},{Kh}(U_m)))=\mathbb{Z}^{2^m}$. It is also known that $\dim\KHI(U_m)=2^{m-1}$. So Theorem \ref{KHI<Kh} is sharp
is this case. Another example is the Hopf link $H$. We have 
$$
Kh(H)\cong R_2/(X_1-X_2)\oplus R_2/(X_1-X_2)
$$
Based on this it is easy to check that $H(K(\mathbf{X},{Kh}(H)))=\mathbb{Z}^8$. On the other hand  
$\dim\KHI(H)=4$. So Theorem \ref{KHI<Kh} is also sharp in this case.
\end{eg}

In the situation that $L$ is a knot (denote it by $K$ now), 
Theorem \ref{KHI<Kh} is closely related to Kronheimer and Mrowka's theorem \cite{KM:Kh-unknot} which says 
\begin{equation}\label{KM-KHI<Kh}
\dim \KHI(K) \le \rank \widetilde{Kh}(K)
\end{equation}
where $\widetilde{Kh}(K)$ is the reduced Khovanov homology of $K$.
In fact, Theorem \ref{KHI<Kh} can be thought as a ``non-reduced'' version of  \eqref{KM-KHI<Kh}.
We also have a reduced version (Theorem \ref{final-inequality}) which is exactly Kronheimer-Mrowka's original theorem when $L$ is a knot.

The instanton knot homology $\KHI(K)$ is a very strong invariant of knots which is a complex vector space equipped with 
a (Alexander) $\mathbb{Z}$-grading. Similar to its Heegaard Floer homology cousin, $\KHI(K)$ detects 
the knot genus and fibered knots \cite{KM:suture}. 
In particular it detects the unknot and the trefoil. Even if we forget the grading, the dimension of $\KHI(K)$ is still enough
to detect the unknot \cite{KM:suture} and the trefoil \cite{BS}. Therefore \eqref{KM-KHI<Kh} can be used to show that
Khovanov homology detects the unknot \cite{KM:Kh-unknot} and the trefoil knots \cite{BS}.  

Kronheimer and Mrowka's proof of \eqref{KM-KHI<Kh} has two steps: 
firstly they defined an invariant $\II^\natural (L)$ using their singular instanton Floer theory,
which is isomorphic to $\KHI(K)$ when $L=K$ is a knot; then they constructed a spectral sequence relating Khovanov homology 
to $\II^\natural(L)$. The advantage of $\II^\natural (L)$ is that it satisfies the unoriented skein exactly triangle: given three links
as in Figure \ref{L210}, there is a 3-cyclic long exact sequence
\begin{equation*}
\cdots\to \II^\natural(L_2) \to \II^\natural(L_1) \to \II^\natural(L_0) \to \II^\natural(L_2) \to \cdots
\end{equation*}
The spectral sequence is obtained by ``iterating'' this exact triangle. The disadvantage of $\II^\natural$ is that even 
though it is isomorphic to $\KHI$ for knots, the Alexander grading of $\KHI$ is lost under this isomorphism.  

\begin{figure}
\centering 
\begin{tikzpicture}
\draw[thick] (1,-1) to (-1,1); \draw[thick,dash pattern=on 1.3cm off 0.25cm] (1,1) to (-1,-1);  \node[below] at (0,-1.45) {$L_2$};
\draw[dashed] (0,0) circle [radius=1.414];

\draw[thick] (2,1)  to [out=315,in=90]  (2.7,0) to [out=270,in=45]    (2,-1);  \node[below] at (3,-1.45) {$L_1$};
\draw[thick] (4,1)  to [out=225,in=90]  (3.3,0) to [out=270,in=135]   (4,-1);
\draw[dashed] (3,0) circle [radius=1.414];

\draw[thick] (5,1)  to [out=315,in=180]  (6,0.3) to [out=0,in=225]   (7,1);
\draw[thick] (5,-1)  to [out=45,in=180]  (6,-0.3) to [out=0,in=135]   (7,-1);  \node[below] at (6,-1.45) {$L_0$};
\draw[dashed] (6,0) circle [radius=1.414];

\end{tikzpicture}
\caption{}\label{L210}
\end{figure} 

If $L$ is not a knot, Kronheimer and Mrowka's spectral sequence still holds. But $\II^\natural (L)$ is not isomorphic to 
$\KHI(L)$ any more.  In the definition of $\II^\natural (L)$, an \emph{earring} (see Definition \ref{earring-def}) is added 
at the base point of $L$ and
$\II^\natural$ is defined as the singular instanton Floer homology of the new link with the earring. 
We want to generalize Kronheimer-Mrowka's definition to obtain a singular instanton invariant which is isomorphic 
to $\KHI(L)$. Before that, we introduce the following definition from \cite{BLS}.
\begin{DEF}\label{pointed-link}
A \emph{pointed} link $(L,\mathbf{p})$ in $S^3$ is a link $L$ together with a collection of marking points 
$\mathbf{p}=\{p_i\}$ on $L$.  The pointed link $(L,\mathbf{p})$ is called \emph{non-degenerate} if 
every component of $L$ contains at least one marking point in $\mathbf{p}$.
\end{DEF}
If we add an earring to each component of $L$, the singular instanton Floer homology of the resulting link turns out to be isomorphic
to $\KHI(L)$. Recall that $\KHI(L)$ is defined as the sutured instanton Floer homology of the link complement with 
a pair of oppositely-oriented meridian sutures on each boundary torus. Motivated by this fact and the definition of
Heegaard knot homology for pointed links in \cite{BL,BLS}, we  define the reduced singular instanton Floer homology
$\II^\natural(L,\mathbf{p})$ by adding an earring to each marking point and taking the singular instanton Floer homology 
when $\mathbf{p}$ is not empty. 
If $(L,\mathbf{p})$ is non-degenerate,
we can take the link complement and add a pair of oppositely-oriented meridian sutures for each point in $\mathbf{p}$,
the sutured instanton Floer homology of the resulting sutured manifold is isomorphic to $\II^\natural(L,\mathbf{p})$ 
(see Proposition \ref{SHI=I-non-degenerate}). 

Since we do not want to have any constraint on $\mathbf{p}$, we define a non-reduced version $\II^\sharp(L,\mathbf{p})$
by adding an unknot with an earring to the link used to defined $\II^\natural(L,\mathbf{p})$. 

Given a pointed link $(L,\mathbf{p})$, its Khovanov homology 
$$Kh(L,\mathbf{p})=H(CKh(L,\mathbf{p}),d_{\mathbf{p}})$$ 
is defined in \cite{BLS}. We define
a variant $Kh'(L,\mathbf{p})$ by modifying the differential $d_\mathbf{p}$ slightly. The modification is minor so that we still have
\begin{equation}
Kh'(L,\mathbf{p};\mathbb{Z}[1/2])\cong Kh(L,\mathbf{p};\mathbb{Z}[1/2])
\end{equation} 
We prove that 
$\II^\sharp$ is related to $Kh'$ by a spectral sequence.
\begin{THE}\label{ss2*}
Let $(L, \mathbf{p})$ be a pointed link in $S^3$ and $\bar{L}$ be the mirror of $L$. Then
there is a spectral sequence whose $E_2$-page is $Kh'(\bar{L},\mathbf{p})$ and which converges to 
$\II^\sharp(L,\mathbf{p})$.
\end{THE}
We also have a spectral sequence for the reduced case (Theorem \ref{ss4}).
Theorem \ref{KHI<Kh} is proved by using the above spectral sequence and some basic homological algebra.

\emph{Acknowledgments.} 
This work is motivated by conversations with Haofei Fan on pointed Heegaard knot homology
and the author would like to thank him for his patience in explaining results
in Heegaard Floer theory to the author. The author thanks Boyu Zhang for helpful suggestions on Proposition \ref{S-RP2}.  

\section{Singular instanton Floer homology}\label{s-inst}
In this section we will give a brief review of the singular instanton Floer homology developed in \cite{KM:Kh-unknot}.

Given a link $L$ in an oriented closed 3-manifold $Y$, the pair $(Y,L)$ can be equipped with an 
an orbifold structure and an orbifold Riemannian metric such that the local stabilizer group $H_x$ is $\mathbb{Z}/2$
if $x\in L$ and trivial otherwise. Let $u$ be a 1-cycle (a collection of arcs and circles) in $Y$ such that
the  $u$ meets $L$ normally at $\partial u=u\cap L$, then $u$ determines an orbifold $SO(3)$ bundle $\check{E}$ on $(Y,L)$ 
(i.e. the \emph{singular bundle data} defined in \cite{KM:Kh-unknot}*{Section 4.2}). For $x\in L$, the  
the local stabilizer group $H_x$ acts on the $SO(3)$ fiber non-trivially. It may happen that 
$\check{E}$ does not extend to all of $Y$ as a topological bundle. But it always extends locally and there are two ways
to extend it near each point $p\in L$. The choices of extensions give us a double cover $L_\Delta \to L$ which has
$w_1=[\partial u]\in H^1(L;\mathbb{Z}/2)$. In particular, $\check{E}$ extends globally if and only if $[\partial u]=0$ 
in $H^1(L;\mathbb{Z}/2)$. The orbifold bundle $\check{E}$ only depends on $u$ on the homology level: suppose 
there is another 1-cycle $v$ with $(v,\partial v)\subset (Y,L)$ such that 
\begin{itemize}
\item $[\partial v]=[\partial u]$ in $H^1(L;\mathbb{Z}/2)$;
\item The previous assumption implies that we can deform $(u,\partial u)$, $(v,\partial v)$ in $(Y,L)$ so that
      $v\cup u$ becomes a closed 1-cycle in $Y$. We require $[u\cup v]=0$ in $H^1(Y;\mathbb{Z}/2)$.
\end{itemize}
Then the orbifold bundles determined by $u$ and $v$ are isomorphic.

The asymptotic holonomy around $L$ of an orbifold connection on $\check{E}$ is an order $2$ element in $SO(3)$. 
Even though $\check{E}$ does not necessarily extend, the gauge group $\mathcal{G}$ of determinant $1$ gauge transformations 
always does. At a point $p\in L$, the gauge tranformations take values 
in a $S^1$-subgroup of $SU(2)$ which preserves the asymptotic holonomy of the orbifold connections. 
We use $\mathcal{A}$ to denote the space of orbifold connections on $\check{E}$ (with a proper Sobolev completion)
and the quotient $\mathcal{B}:=\mathcal{A}/\mathcal{G}$ is the configuration space.

We say the triple $(Y,L,u)$ is \emph{admissible} if there is an embedded surface $\Sigma\subset Y$ such that 
either
\begin{itemize}
  \item $\Sigma$ is disjoint from $K$ and $\omega\cdot \Sigma$ is odd; or
  \item  $\Sigma$ intersects $K$ transversally and $\Sigma\cdot K$ is odd.
\end{itemize}
Given an admissible triple $(Y,L,u)$, the instanton Floer homology
$
\II(Y,L,u)
$
is defined as the Morse homology of the Chern-Simons functional (with a generic perturbation)
\begin{equation*}
CS:\mathcal{B}\to S^1
\end{equation*}

The 4-dimensional case is similar. Given a cobordism $(W,\Sigma,h)$ between two admissible triples
$(Y_0,L_0,u_0)$ and $(Y_1,L_1,u_1)$, we can equip $(W,\Sigma)$ an orbifold structure and an orbifold bundle is determined by
the 2-cycle $h$ (an embedded surface with $\partial h\subset \Sigma$). Counting the number of points in the 0-dimensional
moduli space of anti-self-dual connections on the orbifold bundle, we obtain a morphism
\begin{equation*}
\II(W,\Sigma,h):\II(Y_0,L_0,u_0)\to \II(Y_1,L_1,u_1)
\end{equation*}
which is well-defined up to an overall sign ambiguity.
As in the 3-dimensional case, the orbifold bundle $\check{E}$ on $(W,\Sigma)$ does not always extend and 
the choices of extensions at points in $\Sigma$ give us a double cover $\Sigma_\Delta\to \Sigma$ with $w_1=[\partial h]$.
According  to the discussion in \cite{KM:Kh-unknot}*{Section 4.2}, in a neighborhood 
$\text{Nbh}_{\Sigma}(p)$ of a point $p\in \Sigma\setminus h$, the restriction of the double cover $\Sigma_\Delta$ is not
only trivial but also trivialized. This means we have a preferred extension $E$ of $\check{E}$ near $p$.
We can decompose $E_p$ into the $\pm 1$ eigenspaces $\mathbb{R}$ and $K_p$ of the asymptotic holonomy of orbifold connections 
where $K_p$ is a 2-plane.
Let $N(\Sigma)_p$ be the normal bundle of $\Sigma$ at $p$.
An choice of extension at $p$ determines (and is determined by) 
an identification of $\det (N(\Sigma)_p)$ and $\det (K_p)$ (cf. \cite{KM-surface1}*{Section 2(iv)}). 
A different choice of the extension changes the identification by a sign. 
The surface $\Sigma$ is possibly non-orientable, but we can always pick an orientation at $p$. After so 
$N(\Sigma)_p$ hence $K_p$ are also oriented.

Since the gauge group $\mathcal{G}$ preserves this decomposition and the orientation of $K_p$, 
we can obtain a complex line bundle
\begin{equation*}
\mathbb{K}_p \to \mathcal{B} 
\end{equation*}
There is an equivalent way to construct this line bundle. The framed gauge group $\mathcal{G}_p\subset \mathcal{G}$
consists of gauge transformations restricted to the identity at $p$. The associated complex line bundle of the principal $S^1$-bundle
\begin{equation*}
\tilde{B}:=\mathcal{A}/\mathcal{G}_p \to \mathcal{B} 
\end{equation*}
is just $\mathbb{K}_p$. The moduli space obtained using the framed gauge group $\mathcal{G}_p$ is called the framed 
moduli space which is framed at $p$. It is a $S^1$-bundle over the unframed moduli space. 

The line bundle $\mathbb{K}_p$ is firstly considered in \cite{Kr-ob} in the situation that $\Sigma$ is oriented and
$\check{E}$ extends globally.  We use $\sigma_p$ to denote $e(\mathcal{K}_p)$. It can be represented by a divisor $V_\sigma$.
Using this class, we can define another morphism 
\begin{equation*}
\II(W,\Sigma,h;\sigma_p):\II(Y_0,L_0,u_0)\to \II(Y_1,L_1,u_1)
\end{equation*}
by
\begin{equation*}
\II(W,\Sigma,h;\sigma)(\alpha):=\sum_\beta (\# M_2(W,\Sigma,h;\alpha,\beta)\cap V_\sigma)\cdot \beta
\end{equation*}
where $M_2(W,\Sigma,h;\alpha,\beta)$ denotes the 2-dimensional moduli space with limiting connections $\alpha$ and $\beta$ on the ends.

If we change either the extension at $p$ or the orientation of $\Sigma$ at $p$, the class $\sigma_p$ will be changed by a sign.
If the branched double cover $\pi:\Sigma_\Delta \to \Sigma$ is non-trivial on the component of $\Sigma$ containing $p$, 
then we can pick a loop passing through $p$ whose lifting is an arc joining the two points in $\pi^{-1}(p)$. 
Along this arc, the bundle $\mathbb{K}_p$ is deformed into its dual $\bar{\mathbb{K}}_p$, which implies 
$\sigma_p=-\sigma_p$ hence $2\sigma_p=0$. 

We use $H$ to denote a Hopf link in $S^3$ and $w$ to denote an arc joining the two components of $H$. 
Given a link $L\subset S^3$, its instanton Floer homology is defined in \cite{KM:Kh-unknot} as
\begin{equation*}
\II^\sharp(L):=\II(S^3,L\cup H,w)
\end{equation*}
Given a cobordism $S:L_0\to L_1$ in $S^3$, we have
\begin{equation*}
\II^\sharp(S):=\II(I\times S^3,S\cup I\times H,I\times w): \II^\sharp(L_0) \to \II^\sharp(L_1)
\end{equation*}
where $I=[-1,1]$. In this way, $\II^\sharp$ becomes a functor from the category of links with morphisms the link cobordisms to
the category of abelian groups with morphisms the group homomorphisms modulo a sign.
When $S$ is an oriented cobordism between two oriented links,  $\II^\sharp(S)$ is well-defined \emph{without}
any sign ambiguity. 

Given two links $L_1$ and $L_2$, if either $\II^\sharp(L_1)$ or $\II^\sharp(L_2)$ is torsion-free, then we have
\begin{equation*}
\II^\sharp(L_1\cup L_2)=\II^\sharp(L_1)\otimes \II^\sharp(L_2)
\end{equation*}
Moreover, this isomorphism is natural with respect to split cobordisms (i.e. disjoint union of two cobordisms).

\section{Earrings and exact triangles}\label{earring+triangle}
\begin{DEF}\label{earring-def}
Let $L$ be a link in $S^3$ and $p\in L$ be a point. An \emph{earring} of $L$ at $p$ means a meridian $m$ of $L$ around $p$ together with
an arc $u$ joining $m$ to $L$ at $p$. 
\end{DEF} 

\begin{figure}
\center
\begin{tikzpicture}
\tikzset{
    partial ellipse/.style args={#1:#2:#3}{
        insert path={+ (#1:#3) arc (#1:#2:#3)}
    }
}
\draw[thick] (0,0) [partial ellipse=-265:85: 1.5cm and 1cm] ;
\draw[thick, dash pattern=on 3.9cm off 0.2cm on 100cm] (0,3) to (0,-3);
\draw[red] (0,0) to (1.5,0);

\node[below] at (0,-3.3) {$L\cup m$};

\draw[thick] (4,0) [partial ellipse=-245:65: 1.5cm and 1cm] ;
\draw[thick, dash pattern=on 1.2cm off 0.2cm on 100cm] (4,0.3) to (4,-3);
\draw[thick] (4,3) to (4,1.5); 

\draw[thick] (4,0.3) to [out=90,in=155]  (4+1.5*0.4226, 0.9063);
\draw[thick] (4,1.5) to  [out=-90, in=25]   (4-1.5*0.4226, 0.9063);

\draw[red] (4,0) to  (5.5,0);

\node[below] at (4,-3.3) {$L_1$};
\draw[thick] (8,0) [partial ellipse=-245:65: 1.5cm and 1cm] ;
\draw[thick, dash pattern=on 1.2cm off 0.2cm on 100cm] (8,0.3) to (8,-3);
\draw[thick] (8,3) to (8,1.5); 

\draw[thick] (8,0.3) to [out=90, in=25] (8-1.5*0.4226, 0.9063) ;
\draw[thick] (8,1.5) to [out=-90, in=155] (8+1.5*0.4226, 0.9063);

\draw[red] (8,0) to  (9.5,0);

\node[below] at (8,-3.3) {$L_0$};
\end{tikzpicture}
\caption{Two resolutions of link $L$ with earring $m$. The red arc represents the arc $u$.}\label{L10}
\end{figure}
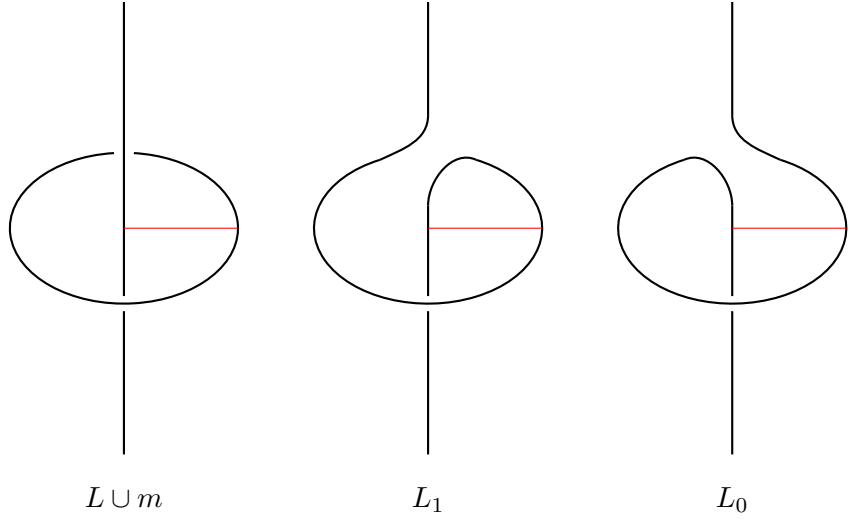

Given a link $L$ with an earring $m$ at $p\in L$,
we have a ``local digram'' for $L\cup m$ near $p$ with two crossings as shown in Figure \ref{L10}.
We can resolve the top crossing by $1$- and $0$- resolutions to obtain two new links $L_1$ and $L_0$. 
Notice that  the two new links are isotopic to the original link $L$ \emph{without} the earring. 
The two copies of arc $u$ in $L_1$ and $L_0$ can be deformed into empty 1-cycles, so we have 
\begin{equation*}
\II(L_i\cup m \cup H, u+w)= \II^\sharp(L),~ i=1,0 
\end{equation*}   
Using Kronheimer-Mrowka's unoriented skein exact triangle \cite{KM:Kh-unknot}*{Section 6}, we have a 3-cyclic exact sequence
\begin{equation}\label{L-L1-L0}
\xymatrix{
\cdots\to \II(L\cup m \cup H ,u+w)\to \II^\sharp(L) \ar[r]^-{f} &\II^\sharp(L) \to \cdots
}
\end{equation}  
The map $f$ is induced by a cobordism from $(L_1,u)$ to $(L_0,u)$. This cobordism can be obtained by attaching a band
to the product cobordism $(I\times L_1,I\times u)$ along $\{1\}\times L_1$ as shown in Figure \ref{L1-band}.
Even though the arc $u$ can be deformed into an empty arc on the two ends, the
2-cycle $I\times u$ represents non-trivial singular bundle data on the cobordism 
$$
(I\times S^3, S): (S^3,L_1)\to (S^3,L_0)
$$  

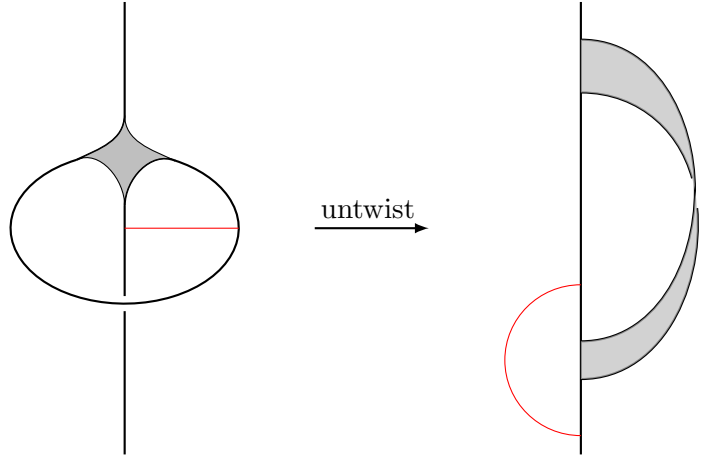
\begin{figure}
\center
\begin{tikzpicture} 
\tikzset{
    partial ellipse/.style args={#1:#2:#3}{
        insert path={+ (#1:#3) arc (#1:#2:#3)}
    }
}

\draw[thick] (4,0) [partial ellipse=-245:64: 1.5cm and 1cm] ;
\draw[thick, dash pattern=on 1.2cm off 0.2cm on 100cm] (4,0.3) to (4,-3);
\draw[thick] (4,3) to (4,1.5); 

\draw[thick] (4,0.3) to [out=90,in=156]  (4+1.5*0.4426, 0.8953);
\draw[thick] (4,1.5) to  [out=-90, in=25]   (4-1.5*0.4226, 0.9063);

\draw[fill=gray!50, opacity=1] (4,0.3) to [out=90,in=156]  (4+1.5*0.4426, 0.8953) to [in=-90, out=155] (4,1.5)
                 to [out=-90, in=25]   (4-1.5*0.4226, 0.9063) to [in=90, out=25]  (4,0.3);

\draw[red] (4,0) to  (5.5,0);

\draw[thick] (10,3) to (10,-3);

\draw[thick] (10,2.5) to [out=0,in=90] (11.5,0.5) to [out=-90, in= 0] (10,-1.5);
\draw[thick,dash pattern=on 2cm off 0.4cm on 100cm] (10,1.8) to [out=0,in=100] (11.5,0.5) to [out=-80,in=0] (10,-2);
\draw[red] (10,-1.75) [partial ellipse=90:270:1cm and 1cm];

\path[fill=gray!50, opacity=0.7] (10,2.5) to [out=0,in=90] (11.5,0.5) to [in=0,out=100] (10,1.8);
\path[fill=gray!50, opacity=0.7] (10,-2) to  [in=-80,out=0] (11.5,0.5) to [out=-90, in= 0] (10,-1.5);

\draw[thick,-latex] (6.5,0) to (8,0);
\node[above] at (7.2,0) {untwist};
\end{tikzpicture}
\caption{Attach a band to $L_1$. Untwist $L_1$ to obtain a straight line attached with a twisted band.}\label{L1-band}
\end{figure}

The topology of $S$ is not hard to understand. If we ignore the embedding, $S$ is just the connected sum of $I\times L$ and
$\mathbb{RP}^2$. Or equivalently, $S$ can be obtained by removing a disk from the product cobordism $I\times L$ and then
attaching a M\"{o}bius band. The restriction of $I\times u$ to $S$ can be deformed into the core circle of the M\"{o}bius band. 

The surface $S$ is a product cobordism away from a neighborhood of $I\times \{p\}$. 
We can untwist $L_1$  to make it locally into a straight line in $\mathbb{R}^2=\mathbb{R}^2\times \{0\}$, which
is depicted in Figure \ref{L1-band}. 
This means the cobordism $S$ in $N(I\times \{p\})$ can be described as attaching a \emph{right-handed half-twisted} band to
product cobordism $I\times I\subset I\times \mathbb{R}^3$ along $\{1\}\times I$. The above discussion implies the following
(cf. \cite{KM:Kh-unknot}*{Lemma 7.2} and the Hopf link example in \cite{KM-filtration}*{Section 11}). 

\begin{PR}\label{S-RP2}
If $S$ be the cobordism as above, then we have
\begin{equation*}
(I\times S^3, S)=(I\times S^3, V)\# (S^4, \mathbb{RP}^2)
\end{equation*} 
where $V$ is isotopic to the product cobordism $I\times L$ and 
the $\mathbb{RP}^2$ is standardly embedded in $S^4$ with self-intersection $-2$. Moreover, the restriction
of the 2-cycle $I\times u$ to $\mathbb{RP}^2$ is non-trivial in $H_1(\mathbb{RP}^2; \mathbb{Z}/2)$. 
\end{PR}

We want to use a gluing argument to study 
$$
f=\II(I\times S^3, S\cup I\times H,I\times u +I\times w)
$$
Before that, we need to know some facts on the minimal energy moduli space $M(S^4,\mathbb{RP}^2,\eta)$ where 
$ \eta \subset S^4$ is a 2-cycle and $\partial \eta\subset \mathbb{RP}^2$ is non-trivial in  $H_1(\mathbb{RP}^2; \mathbb{Z}/2)$.
This space has been studied in \cite{KM:Kh-unknot}*{Section 2.7}. It consists of the unique flat connection $[B]$ on 
$S^4\setminus \mathbb{RP}^2$ with the holonomy around $\mathbb{RP}^2$ to be an order $2$ element in $SO(3)$. 
As a reducible connection, the stabilizer $\Gamma_B$ is $S^1$. 

The index of the deformation complex of $B$ is $-3$, which can be calculated with the help of the
branched double cover of $(S^4,\mathbb{RP}^2)$.
The branched double cover of $(S^4,\mathbb{RP}^2)$ with $\mathbb{RP}^2\cdot \mathbb{RP}^2=-2$ is 
the complex projective plane $\mathbb{CP}^2$ and the covering involution is the complex conjugation. 
The lifting of $[B]$ to $\mathbb{CP}^2$ is the trivial flat connection on the rank $3$ trivial bundle 
$\underline{\mathbb{R}}\oplus \underline{\mathbb{C}}$. We also assume $\mathbb{CP}^2$ is equipped with the standard metric.
 The covering involution acts trivially on $\underline{\mathbb{R}}$
and acts as $-1$ on $\underline{\mathbb{C}}$. The obstruction space $H^2_B$ can be calculated as the invariant part of 
$H^2_+(\mathbb{CP}^2;\underline{\mathbb{R}}\oplus \underline{\mathbb{C}})$. We have 
$$H^2_+(\mathbb{CP}^2;\mathbb{R})=H^2(\mathbb{CP}^2;\mathbb{R})$$ 
is generated by the standard K\"{a}hler form  $\omega$, which satisfies $\bar{\omega}=-\omega$. Therefore the invariant part
of $H^2_+(\mathbb{CP}^2;\underline{\mathbb{R}}\oplus \underline{\mathbb{C}})$ consists of $\omega\otimes s$ where $s$ is
an arbitrary constant section of $\underline{\mathbb{C}}$. Similarly, 
$H^1(\mathbb{CP}^2;\underline{\mathbb{R}}\oplus \underline{\mathbb{C}})=0$ implies $H^1_B=0$. 

\begin{LE}\label{key}
Let $(S^4,\mathbb{RP}^2,\eta)$ be given as above. 
Suppose $(X,\Sigma, h)$ be an admissible cobordism and $p$ is a point on $\Sigma\setminus h$. Then we have
\begin{equation*}
\II((X,\Sigma,h)\#(S^4,\mathbb{RP}^2,\eta))=\pm\II(X,\Sigma,h;\sigma_p)
\end{equation*}
where the connected sum is taken at $p$ and a point $q\in \mathbb{RP}^2\setminus \eta$.
\end{LE}
\begin{proof}
To simplify the notation in the discussion, we assume $(X,\Sigma)$ is a closed pair so that 
we need to prove 
\begin{equation*}
D((X,\Sigma,h)\#(S^4,\mathbb{RP}^2,\eta))=\pm D(X,\Sigma,h)(\sigma_p)
\end{equation*}
where the left-hand-side is the Donaldson invariant obtained by counting points in the
 0-dimensional moduli spaces and the right-hand-side
is  obtained by cutting down the 2-dimensional moduli spaces using $V_\sigma$ and counting points.

We use $M_2(X,\Sigma,h)$ to denote the 2-dimensional moduli space. It is compact since there is no bubble for dimension reason.
Let $\widetilde{M}_2(X,\Sigma,h)$ be
corresponding framed moduli space which is framed at $p$. There is an $S^1$-action on  
$\widetilde{M}_2(X,\Sigma,h)$ which makes it into a principal $S^1$ bundle over $M_2(X,\Sigma,h)$.

Since the minimal energy moduli space $M(S^4,\mathbb{RP}^2,\eta)$ consists of a single element $[B]$ which is a $S^1$-reducible connection,
the $S^1$-action on $\widetilde{M}(S^4,\mathbb{RP}^2,\eta)$ is trivial. Here we use
$\widetilde{M}(S^4,\mathbb{RP}^2,\eta)$ to denote the moduli space framed at $q$. The $S^1$-action comes from 
the quotient $\mathcal{G}/\mathcal{G}_q$ where $\mathcal{G}$ is the gauge group and $\mathcal{G}_q$ is the framed gauge group.
In particular, the $S^1$ acts on $E_q=\mathbb{R}\oplus \mathbb{C}$ in the 
standard way: the action is trivial on $\mathbb{R}$ and standard on $\mathbb{C}$. Combined with the description of
the obstruction space $H^2_B=\mathbb{C}$ we know that $S^1$ acts on it in the standard way. 

Dimension counting and gluing theory for obstructed cases (see \cite{DK}*{Proposition 9.3.13}) show that
the 0-dimensional moduli space 
$$M((X,\Sigma,h)\#(S^4,\mathbb{RP}^2,\eta))$$ 
is the zero set of a section of the complex line bundle
\begin{equation*}
\mathbb{K}=\widetilde{M}_2(X,\Sigma,h)\times _{S^1} H^2_B \to {M}_2(X,\Sigma,h)
\end{equation*}
According to Section \ref{s-inst}, this complex line bundle is exactly the bundle used to define $\sigma_p$. So we have
\begin{equation*}
\# M((X,\Sigma,h)\#(S^4,\mathbb{RP}^2,\eta)) = e(\mathbb{K}) [{M}_2(X,\Sigma,h)]= \pm D(X,\Sigma,h)(\sigma_p)
\end{equation*}
\end{proof}
A very similar result is discussed in \cite{Kr-ob}*{Proposition 5.1}: using the pair $(S^4,T^2)$ (with trivial bundle on it)
instead of $(S^4,\mathbb{RP}^2,\eta)$, Lemma \ref{key} still holds. More precisely, we have
\begin{equation}\label{sigma-T2}
\II((X,\Sigma,h)\#(S^4,T^2))=\II(X,\Sigma,h;\sigma_p)
\end{equation} 
This time there is no plus-minus sign since the sign ambiguities on the two sides cancel with each other. 
Even though the situations are different, the 
minimal energy moduli spaces $M(S^4,\mathbb{RP}^2,\eta)$ and $M(S^4,T^2)$ share some common formal properties:
\begin{itemize}
\item They consist of the unique $S^1$-reducible flat connection with deformation complex of index $-3$;
\item The elements in the obstruction space 
can be described as the tensor product  of the unique anti-self-dual harmonic 2-form and a section of a trivial complex line bundle
on the branched double cover.
\end{itemize}  
Indeed the proof of Lemma \ref{key} can be carried over almost verbatim from the proof of \cite{Kr-ob}*{Proposition 5.1} since
only those formal properties play a role in the proof. 

Given any admissible triple $(Y,L,u)$ and pick a point $p\in L\setminus u$ with an choice of orientation of $L$ at $p$,
there is an operator
\begin{equation*}
\sigma_p: \II(Y,L,u)\to \II(Y,L,u)
\end{equation*} 
defined in \cite{KM:Kh-unknot}*{Section 8.3} as the map
\begin{equation*}
\sigma_p:= \II(I\times Y, I\times L\# T^2, I\times u)
\end{equation*} 
where the connected sum is taken at $(0,p)\in I\times L$. By \eqref{sigma-T2}, we have
\begin{equation}\label{sigma=sigma}
\sigma_p= \II(I\times Y, I\times L, I\times u;\sigma_p)
\end{equation} 
Notice that the two $\sigma_p$'s in the above equality represent different things: the first $\sigma_p$ is a map and 
the second $\sigma_p$ is a cohomology class in the configuration space. We abuse the notation essentially because of \eqref{sigma=sigma}.
Now by Proposition \ref{S-RP2}, Lemma \ref{key} and \eqref{sigma=sigma}, we have 
\begin{PR}\label{f=sigma}
The map  $f$ in \eqref{L-L1-L0} is equal to $\pm \sigma_p$.
\end{PR}

\section{Instanton Floer homology for pointed links}
We use $U_k$ to denote the unlink with $k$ components in $S^3$. We fix a generator 
\begin{equation*}
\mathbf{u}_0\in \II(U_0)\cong \mathbb{Z}
\end{equation*}
where $U_0$ is the empty link.
Let $D^\pm$ be standard disks in $I\times S^3$ which give oriented cobordisms from $U_0$ to $U_1$ and $U_1$ to $U_0$ respectively.
Similarly, let $\Sigma^\pm$ be standard punctured tori in $I\times S^3$
which give oriented cobordisms from $U_0$ to $U_1$ and $U_1$ to $U_0$ respectively.
\begin{PR}[{{\cite{KM:Kh-unknot}*{Section 8}}}]\label{unknot-gen}
There are generators $\mathbf{v}_+$ and $\mathbf{v}_-$ for $\II(U_1)\cong \mathbb{Z}^2$ characterized by
\begin{equation*}
  \II(D^+)(\mathbf{u}_0)=\mathbf{v}_+, ~~\II(D^-)(\mathbf{v}_-) = \mathbf{u}_0
\end{equation*}
We also have
\begin{equation*}
  \II(\Sigma^+)(\mathbf{u}_0) = 2\mathbf{v}_-,~~\II(\Sigma^-)(\mathbf{v}_+) = 2\mathbf{u}_0
\end{equation*}
In terms of those generators, the map induced by the pair of pants cobordism from $U_2$ to $U_1$ is given by
\begin{align*}
  \mathbf{v}_+ \otimes \mathbf{v}_+ &\mapsto \mathbf{v}_+,~~ \mathbf{v}_- \otimes \mathbf{v}_+ \mapsto \mathbf{v}_- \\
   \mathbf{v}_+ \otimes \mathbf{v}_- &\mapsto \mathbf{v}_-,~~ \mathbf{v}_-\otimes \mathbf{v}_- \mapsto  0
\end{align*}
The map induced by the pair of pants cobordism from $U_1$ to $U_2$ is given by
\begin{align*}
  \mathbf{v}_+ &\mapsto \mathbf{v}_+\otimes \mathbf{v}_- + \mathbf{v}_-\otimes \mathbf{v}_+  \\
  \mathbf{v}_- &\mapsto \mathbf{v}_-\otimes \mathbf{v}_-
\end{align*}
\end{PR}
If $p$ is a point on $U_1$, the cobordism used to define $\sigma_p$ can be described as attaching $\Sigma_+$ along a circle on the incoming end
of a pair of pants cobordism. The functoriality of $\II^\sharp$ and Proposition \ref{unknot-gen}
imply $\sigma_p:\II^\sharp(U_1)\to \II^\sharp(U_1)$ is given by
\begin{equation}\label{sigma-U0}
\mathbf{v}_+\to 2\mathbf{v}_-,~~\mathbf{v}_- \to 0 
\end{equation}
Another application of Proposition \ref{unknot-gen} is the following.
\begin{PR}\label{sigma-Z2-vanishing}
Suppose $L$ is an arbitrary link in $S^3$ and $u$ is a 1-cycle in $S^3$ with $\partial u\subset L$. Take a point $p\in L\setminus u$.
 Then we have the operator 
\begin{equation*}
\sigma_p=\II^\sharp(I\times L\# T^2,I\times u;\mathbb{Z}/2)=\II(I\times S^3, I\times L\# T^2,I\times u\cup I\times w;\mathbb{Z}/2)
\end{equation*}
on $\II^\sharp(L,u;\mathbb{Z}/2)$ is zero.
\end{PR}
\begin{proof}
It follows from the equality
\begin{equation*}
  \II(\Sigma^+)(\mathbf{u}_0) = 2\mathbf{v}_-
\end{equation*}
in Proposition \ref{unknot-gen} and the functoriality of $\II^\sharp$.
\end{proof}

\begin{DEF}
Given a pointed link  $(L,\mathbf{p})$ in $S^3$, we add an earring $(m_i,u_i)$ at each mark point $p_i$ ($1\le i \le l$). 
The unreduced instanton Floer homology
of $(L,\mathbf{p})$ is defined as
\begin{equation*}
\II^\sharp(L,\mathbf{p}):=\II(S^3,L\cup H \cup \bigcup_i m_i, w+\sum_i u_i)
\end{equation*}
If $\mathbf{p}\neq \emptyset$, 
the reduced instanton Floer homology
of $(L,\mathbf{p})$ is defined as
\begin{equation*}
\II^\natural(L,\mathbf{p}):=\II(S^3,L\cup \bigcup_i m_i, \sum_i u_i)
\end{equation*} 
\end{DEF}

Suppose $(L,\mathbf{p}=\{p_i\}_{1\le i \le l})$ is a pointed link in $S^3$ and $q\in L\setminus \mathbf{p}$. Similar to \eqref{L-L1-L0},
we have a 3-cyclic exact sequence
\begin{equation}\label{marking-p+q}
\xymatrix{
\cdots\to \II^\sharp(L,\mathbf{p}\cup\{q\})\to \II^\sharp(L,\mathbf{p}) \ar[r]^-{\sigma_q} &\II^\sharp(L,\mathbf{p}) \to \cdots
}
\end{equation}
By Proposition \ref{sigma-Z2-vanishing}, the map $\sigma_q$ is zero if we use $\mathbb{Z}/2$-coefficients. By induction, 
\eqref{marking-p+q} implies the following.
\begin{PR}\label{I-Lp-Z2}
Suppose $(L,\mathbf{p})$ is a pointed link in $S^3$. The we have
\begin{equation*}
\II^\sharp(L,\mathbf{p};\mathbb{Z}/2)\cong (\mathbb{Z}/2)^{\oplus 2^{|\mathbf{p}|}}\otimes_{\mathbb{Z}/2}\II^\sharp(L;\mathbb{Z}/2)
\end{equation*}
\end{PR}
If we also assume $q$ lies in a component of $L$ with odd marking points, then the orbifold bundle on $(S^3, L\cup \bigcup_i m_i)$
does not extend on the this component.  
The discussion in Section \ref{s-inst} tells us  that $2\sigma_q=0$. If we use $\mathbb{Z}[1/2]$-coefficients, then $\sigma_q=0$. 
On the other hand, $\sigma_q$ is always $0$ in $\mathbb{Z}/2$-coefficients. So we have 
\begin{PR}\label{p+q=2I}
Suppose $(L,\mathbf{p})$ is a pointed link in $S^3$ and $q\in L\setminus \mathbf{q}$ lies in a component of $L$ with odd marking points.
Let $\mathbb{F}$ be any field. Then we have
\begin{equation*}
\II^\sharp(L,\mathbf{p}\cup\{q\};\mathbb{F})\cong \mathbb{F}^2\otimes_{\mathbb{F}}  \II^\sharp(L,\mathbf{p};\mathbb{F})
\end{equation*}
\end{PR}
With a little bit more work, we can derive a stronger result.
\begin{PR}\label{p+q=2I-Z}
Suppose $(L,\mathbf{p})$ is a pointed link in $S^3$ and $q\in L\setminus \mathbf{q}$ lies in a component of $L$ with at 
least one marking points. Then we have $\sigma_q=0$ on $\II^\sharp(L,\mathbf{p})$ and
\begin{equation*}
\II^\sharp(L,\mathbf{p}\cup\{q\})\cong \mathbb{Z}^2\otimes  \II^\sharp(L,\mathbf{p})
\end{equation*}
We also have 
$\sigma_q=0$ on $\II^\natural(L,\mathbf{p})$ and
\begin{equation*}
\II^\natural(L,\mathbf{p}\cup\{q\})\cong \mathbb{Z}^2\otimes  \II^\natural(L,\mathbf{p})
\end{equation*}
\end{PR}
\begin{proof}
Let $p_1$ be a marking point on the component of $L$ where we add $q$.
Instead of adding one marking point $q$, we add two points $q_1$ and $q_2$ and suppose 
$p_1,q_1,q_2$ is consecutive on the component. Denote the earrings assigned to $p_i\in \mathbf{p}$ by $(m_i,u_i)$
and the earrings assigned to $q_1,q_2$ by $(m_1',u_1'), (m_2',u_2')$.

Moving $m_1,m_1',m_2'$ close to each other and take a tubular neighborhood
$N(m_1)$ which contains $m_1,m_1', m_2'$. Notice that the two 1-cycles $u_1', u_2'$ 
can be deformed into an arc $u_{12}'$ joining $m_1'$ and $m_2'$.
Consider another admissible triple $(S^1\times S^2, K_1\cup K_2,v)$ where
$K_1=S^1\times \{r_1\}$ and $K_s=S^1\times \{r_2\}$ are two vertical circles and $v_{12}$ is a small arc joining $K_1$ and $K_2$. 
The representation variety of the triple $(S^1\times S^2, K_1\cup K_2,v_{12})$ consists of a single non-degenerate 
flat connection, hence its Floer homology is just $\mathbb{Z}$.

Given two tori $T_1,T_2$ in a three manifold $Y$, an excision of $Y$ along $T_1, T_2$ means we cut $Y$ along $T_1$ and $T_2$
and re-glue using a diffeomorphism between $T_1$ and $T_2$.
Do excision to
\begin{equation*}
(S^3, L\cup H\cup m_1'\cup m_2'\cup \bigcup_{i=1}^l m_i, w+u_{12}'+\sum_{i=1}^l u_i) \cup (S^1\times S^2, K_1\cup K_2,v_{12})
\end{equation*}
along tori $\partial N(m_1)$ and $\partial N(K_1)$, we obtain another pair of admissible triples
\begin{equation*}
(S^3, L\cup H\cup \bigcup_{i=1}^{l} m_i, w+\sum_{i=1}^l u_i) \cup (S^1\times S^2, K_1\cup K_2\cup K_3\cup K_4,v_{12}+v_{34})
\end{equation*}
where $K_i$'s are vertical circles and $v_{ij}$ is an arc joining $K_i, K_j$.
The excision theorem (\cite{KM:Kh-unknot}*{Theorem 5.6}) 
shows that this excision does not change the instanton Floer homology, which means
\begin{equation}\label{T1111}
\II^\sharp({L},{\mathbf{p}}\cup \{q_1,q_2\})\cong \II^\sharp({L},{\mathbf{p}})
\otimes \II(S^1\times S^2, K_1\cup K_2\cup K_3\cup K_4,v_{12}+v_{34})
\end{equation}
 Do excision twice to 
\begin{equation*}
(S^1\times S^2, K_1\cup K_2\cup K_3\cup K_4,v_{12}+v_{34})
\end{equation*}
along $(\partial N(K_1),\partial N(K_2))$ and $(\partial N(K_3),\partial N(K_4))$, we obtain
 two copies of $(S^1\times S^2, K_1\cup K_2,v_{12})$ and
$(S^1\times \Sigma_2,v')$ where $\Sigma_2$ is a genus 2 surface and $v'$ is a homologically non-trivial 1-cycle on a $\Sigma_2$-slice.
The excision theorem shows that
\begin{equation}\label{T2222}
\II(S^1\times S^2, K_1\cup K_2\cup K_3\cup K_4,v_{12}+v_{34})^{\oplus 2}\cong \II(S^1\times \Sigma_2,v')
\end{equation}
It is know that $\II(S^1\times \Sigma_2,v')=\mathbb{Z}^8$ 
(See \cite{DB:sur-rel}*{Proposition 1.15} or \cite{KM-Alex}*{Lemma 3.2} and the succeeding discussion). So 
\eqref{T1111} and \eqref{T2222} imply 
\begin{equation*}
\II^\sharp({L},{\mathbf{p}}\cup\{q_1,q_2\})\cong \II^\sharp({L},{\mathbf{p}})
\otimes \mathbb{Z}^4
\end{equation*}
On the other hand, $\II^\sharp({L},{\mathbf{p}}\cup\{q_1,q_2\})$ can be obtained by applying the exact sequence 
\eqref{marking-p+q} twice. Therefore we have $\sigma_{q_1}=\sigma_{q_2}=0$ and 
\begin{equation*}
\II^\sharp(L,\mathbf{p}\cup\{q_i\})\cong \mathbb{Z}^2\otimes  \II^\sharp(L,\mathbf{p})
\end{equation*}
The proof for the $\II^\natural$ case is similar.
\end{proof}

 Suppose $L$ is oriented and a diagram $D$ with $n$ crossings has been chosen. We also assume the marking points $\mathbf{p}$
 are away from the crossings.
 Denote the set of crossings for $D$ by $N$. 
 We can pick a diagram $D'$ for $L':=L\cup\bigcup_i m_i$
 which is the same as $D$ on the $L$ component and looks like Figure \ref{L10} for each earring $m_i$. Besides the 
 $n$ crossings from $D$, we also pick $l$ crossings $M=\{c_i\}_{1\le i \le l}$ where $c_i$ is a crossing for $m_i$
 as in Figure \ref{L10}.  The set $N':=N\cup M$ is not a complete list of crossings for $D'$. But for any $v:N'\to \{0,1\}$,
 if we resolve the crossings in $N'$ by 0- or 1- resolutions determined by $v$, the resulting resolved link $L'_v$ is 
 still an unlink ($N'$ is called a \emph{pseudo-diagram} in \cite{KM-filtration}). 
 Given $v,z\in \{0,1\}^{N'}$ and $v\ge z$, there is a skein cobordism $S_{vz}:L'_v\to L'_z$.  
 When $v|_M\neq z|_M$, we also have a non-trivial 2-cycle 
\begin{equation*}
 h_{vz}:=\sum_{1\le i \le l} (v(c_i)-z(c_i))\cdot I\times u_i
\end{equation*}
on $S_{vz}$ as in Section \ref{earring+triangle}. 

According to \cite{KM:Kh-unknot}*{Corollary 6.8}, we have the following.
\begin{PR}\label{ss1}
There is a spectral sequence converging to $\II^\sharp(L,\mathbf{p})$ whose $E_1$-page is
\begin{equation}\label{cube-L'}
\bigoplus_{v\in  \{0,1\}^{N'}} \II^\sharp(L'_v)
\end{equation}
and the differential is the sum of
\begin{equation*}
\pm \II^\sharp(S_{vz},h_{vz}):=\pm \II(I\times S^3, S_{vz}\cup I\times H, h_{vz}+I\times w)
\end{equation*}
with a proper choice of signs where $vz$ ($v\ge z$) is an edge of the cube $\{0,1\}^{N'}$. 
\end{PR}
If $v|_M=z|_M$, then $S_{vz}$ is an oriented cobordism and $\II^\sharp(S_{vz})$ is well-defined without any sign ambiguity.
As addressed in Section \ref{earring+triangle}, resolving the crossings from the earrings
gives us the original link $L$. The resolution $L'_v$ is the same as $L_{z|_N}$. 
 All the links $L_v'$ and cobordisms $S_{vz}$ can be
oriented consistently so that $\partial S_{vz}=L_z'-L_v'$:
take a checkerboard coloring of the regions of the diagram $D$ and orient $L_v'$ 
by the boundary orientation of the black region away from the smoothings.

If $v,z$ differ at a crossing $c\in M$, then $S_{vz}$ is non-orientable (exactly the situation in Proposition \ref{S-RP2}) 
and  $\II^\sharp(S_{vz},h_{vz})$ is only well-defined up to a sign. By Proposition \ref{f=sigma},
\begin{equation*}
\II^\sharp(S_{vz},h_{vz})=\pm \sigma_p
\end{equation*}
We fix the sign of $\II^\sharp(S_{vz},h_{vz})$ by requiring the above equality holds without the plus-minus sign.

We also fix a complete order on the crossing set $N'$ such that elements in $M$ are greater than elements in $N$
and the order on $M$ coincides with the index-order for $c_i\in M$. 
By  \cite{KM:Kh-unknot}*{Corollary 6.8},  the differential on the $E_1$-page is the sum of
\begin{equation*}
\tilde{\delta}(v,z)\II^\sharp(S_{vz},h_{vz})
\end{equation*}
where 
\begin{equation*}
\tilde{\delta}(v,z):=\sum_{c\ge c_0} v(c), ~~ v(c_0)-z(c_0)=1
\end{equation*}
Fix an element $s\in \{0,1\}^M$, we have a 
sub-$n$-cube $cube(s)$ of the $(n+l)$-cube $\{0,1\}^{N'}$ consists of vertices $v\in \{0,1\}^{N'}$ such that
$v|_M=s$. Since $L'_v\cong L_{v|_N}$, the component
\begin{equation*}
C_s:=\bigoplus_{v\in  cube(s)} \II^\sharp(L'_v)
\end{equation*}
of \eqref{cube-L'} is the same as 
\begin{equation}\label{cube-L}
\bigoplus_{v\in  cube(s)} \II^\sharp(L_{v|_N})
\end{equation}
According to \cite{KM:Kh-unknot}*{Section 8}, \eqref{cube-L} can be identified with the Khovanov complex $CKh(\bar{L})$ of 
the mirror of $L$ and the differential on it is $(-1)^{\sum_{c\in M}s(c)} d_{Kh}$. 
Pick $c_i\in M$,
suppose $v=z$ on $N'\setminus \{c_i\}$ and $v(c_i)-z(c_i)=1$. Then there is also a differential 
\begin{equation}\label{vz-sigma}
 (-1)^{\sum_{j\ge i}v(c_j)}\sigma_{p_i}:\II^\sharp(L'_v)\to \II^\sharp(L'_z)
\end{equation}
Fix $s\ge s'$ in $M$ and $s,s'$ only differ at $i$. Collect all the maps in \eqref{vz-sigma} for $v|_{M}=s$, we obtain
an anti-chain map
\begin{equation*}
(-1)^{\sum_{j\ge i}s(c_j)} \sigma_{p_i} : C_s \to C_{s'}
\end{equation*}
In summary, $E_1$-page is 
\begin{equation}\label{cube-Cs}
E_1=\bigoplus_{s\in \{0,1\}^M}C_s
\end{equation}
where each $C_s$ is a copy of the Khovanov chain complex $CKh(\bar{L})$ and the differential is the sum of
\begin{equation}\label{d-Cs}
(-1)^{\sum_{c\in M}s(c)} d_{Kh}:C_s\to C_s,~~ (-1)^{\sum_{j\ge i}s(c_j)} \sigma_{p_i} : C_s \to C_{s'} 
\end{equation}
where $s\ge s'$ and $s,s'$ only differ at $c_i\in M$.

In \cite{BLS}, Baldwin, Levine and Sarkar defined Khovanov homology for pointed links. 
It turns out that the $E_1$-page \eqref{cube-Cs} 
is isomorphic to the chain complex of a variant of their Khovanov homology for pointed links.
Before proving this, we need to review their definition. 

Let $(L,\mathbf{p})$ be a pointed link as before.
 Let $V:=\mathbb{Z}\{\mathbf{v}_+,\mathbf{v}_-\}$ be the free abelian group with two generators. 
For each $v\in \{0,1\}^{N}$, there is a component of the Khovanov chain complex $CKh(L)$ defined by
$V^{\otimes |L_v|}$ where $|L_v|$ denotes the number of components of $L_v$. Since $L_v$ is merely a collection of circles in the plane,
$V^{\otimes |L_v|}$ can be understood as assigning a copy of $V$ to each circle in $L_v$.
Pick a marking point $p_i\in \mathbf{p}$. There is a unique circle in $L_v$ containing $p$. 
Defining a map
\begin{equation*}
X_{p_i}: V^{\otimes |L_v|} \to V^{\otimes |L_v|}
\end{equation*}
by requiring
\begin{equation}\label{CKh-Xp}
X_{p_i}(\mathbf{v}_+)=\mathbf{v_-},~~ X_{p_i}(\mathbf{v}_-)=0
\end{equation}
on the copy of $V$ assigned to the circle containing $p_i$ and $X_{p_i}$ is the identify map on other copies of $V$'s.
In this way, we obtain a chain map $X_{p_i}$ on 
\begin{equation*}
CKh(L):=\bigoplus_{v\in \{0,1\}^N} V^{\otimes |L_v|}
\end{equation*}
It is clear that $X_{p_i}^2=0$.
The map induced by $X_{p_i}$ on $Kh(L)$ does not depend on the choice of diagram of $L$ and only depend on the component of 
$L$ where $p_i$ lies according to \cite{Kh-module}. Suppose $L$ has $m$ components, then $Kh(L)$ is equipped with a
\begin{equation}\label{Rm}
R_m=\mathbb{Z}[X_1,X_2,\cdots, X_m]/(X_1^2,\cdots, X_m^2)
\end{equation}
module structure. 

Let $\Lambda_{\mathbf{p}}$ be the exterior algebra generated by $\{y_i|i=1,\cdots, l\}$ over $\mathbb{Z}$. Define
\begin{equation*}
CKh(L,\mathbf{p}):=\Lambda_{\mathbf{p}}\otimes CKh(L)
\end{equation*}
with differential 
\begin{align}
&d_{\mathbf{p}}(y_{i_1}\wedge\cdots \wedge y_{i_k}\otimes b )= \nonumber \\ 
&(-1)^k y_{i_1}\wedge\cdots \wedge y_{i_k}\otimes d_{Kh}(b)
 +\sum_{i=1}^l y_i\wedge y_{i_1}\wedge\cdots \wedge y_{i_k}\otimes X_{p_i}(b) \label{d-p}
\end{align}
The homology $Kh(L,\mathbf{p})=H(CKh(L,\mathbf{p}),d_\mathbf{p})$  
is a well-defined invariant for $(L,\mathbf{p})$.

We want to vary the definition slightly. We define a new differential by
\begin{align}
&d'_{\mathbf{p}}(y_{i_1}\wedge\cdots \wedge y_{i_k}\otimes b )= \nonumber \\ 
&(-1)^k y_{i_1}\wedge\cdots \wedge y_{i_k}\otimes d_{Kh}(b)
 +\sum_{i=1}^l y_i\wedge y_{i_1}\wedge\cdots \wedge y_{i_k}\otimes 2X_{p_i}(b) \label{d'-p}
\end{align}
Denote the new homology by $Kh'(L,\mathbf{p})=H(CKh(L,\mathbf{p}),d'_\mathbf{p})$. Clearly we have
\begin{equation}\label{Kh=Kh'}
Kh'(L,\mathbf{p};\mathbb{Z}[1/2])\cong Kh(L,\mathbf{p};\mathbb{Z}[1/2])
\end{equation}
But the two homologies could be different with $\mathbb{Z}$ or $\mathbb{Z}/2$ coefficients.

\begin{PR}
The $E_1$-page \eqref{cube-Cs} of the spectral sequence in Proposition \ref{ss1} is isomorphic to 
$(CKh(\bar{L},\mathbf{p}),d'_{\mathbf{p}})$.
\end{PR}
\begin{proof}
We have an identification $C_s=CKh(\bar{L})$. Define $\bar{s}=1-s$ and $|\bar{s}|=\sum_{c_i\in M} \bar{s}(c_i)$.
We define a map
\begin{equation*}
\bigoplus_{s\in \{0,1\}^M}
F:C_s\to CKh(\bar{L},\mathbf{p})=\bigoplus_{i_1< \cdots <i_k} \mathbb{Z} \{y_{i_1}\wedge \cdots \wedge y_{i_k}\}\otimes  CKh(\bar{L})
\end{equation*}
by
\begin{equation*}
b \mapsto  (-1)^{\frac{|\bar{s}|(|\bar{s}|+1)}{2}+\sum_k \bar{s}(c_{2k+1})} 
y_1^{\bar{s}(c_1)}\wedge\cdots \wedge y_l^{\bar{s}(c_l)}   \otimes b
\end{equation*}
It is clear $F$ is a group isomorphism.
Using \eqref{d-Cs}, \eqref{sigma-U0}, \eqref{CKh-Xp} and \eqref{d'-p}, it is straightforward to verify that
\begin{equation*}
F:(E_1,d_1)\to (CKh(L,\mathbf{p}),d'_{\mathbf{p}})
\end{equation*}
is a chain map when $l$ is even and is an anti-chain map when $l$ is odd. When $l$ is odd, $F$ can be made into a chain map
by composing with an anti-automorphism on $CKh(L,\mathbf{p})$. 
\end{proof}

Now we can obtain our main result.
\begin{THE}\label{ss2}
Let $(L, \mathbf{p})$ be a pointed link in $S^3$ and $\bar{L}$ be the mirror of $L$. Then
there is a spectral sequence whose $E_2$-page is $Kh'(\bar{L},\mathbf{p})$ and which converges to 
$\II^\sharp(L,\mathbf{p})$.
\end{THE}

\begin{eg}
Let $p$ be a point on the unknot $U_1$. We have
\begin{equation*}
Kh(U_1,\{p\})\cong \mathbb{Z}^2, ~~Kh'(U_1,\{p\})\cong \II^\sharp(U_1,\{p\})\cong \mathbb{Z}^2\oplus \mathbb{Z}/2
\end{equation*}
The Floer homology  $\II^\sharp(U_1,\{p\})$ can be calculated using \eqref{marking-p+q}.
\end{eg}

Let $p_0 \in L\setminus\mathbf{p}$ be a base point. It can be thought as a distinguished marking point in $\mathbf{p}\cup\{p_0\}$.
We can define the reduced Khovanov homology for pointed link $(L,\mathbf{p})$ with base point $p_0$ in the following way. 
Instead of using the Khovanov chain complex $CKh(L)$, we use the reduced Khovanov chain complex
\begin{equation*}
\widetilde{CKh}(L,p_0)=\ker X_{p_0} \subset CKh(L) 
\end{equation*}
We define
\begin{equation*}
\widetilde{CKh}(L,\mathbf{p},p_0):=\Lambda_{\mathbf{p}}\otimes CKh(L,p_0)
\end{equation*}
whose equipped with differentials $d_{\mathbf{p}}$ and $d_{\mathbf{p}}'$ as in \eqref{d-p} and $\eqref{d'-p}$ with $d_{Kh}$ replaced
by $d_{\widetilde{Kh}}$. Now we can define the reduced Khovanov homology (and its variant) of $(L,\mathbf{p})$ by
\begin{equation*}
\widetilde{Kh}(L,\mathbf{p},p_0)= H(\widetilde{CKh}(L,\mathbf{p},p_0),d_{\mathbf{p}}),~~
\widetilde{Kh}'(L,\mathbf{p},p_0)= H(\widetilde{CKh}(L,\mathbf{p},p_0),d_{\mathbf{p}}')
\end{equation*}
We still have
\begin{equation}\label{Kh=Kh'-reduced}
\widetilde{Kh}'(L,\mathbf{p},p_0;\mathbb{Z}[1/2])\cong \widetilde{Kh}(L,\mathbf{p},p_0;\mathbb{Z}[1/2])
\end{equation}
Similar arguments as before give us the following.
\begin{THE}\label{ss4}
Let $(L, \mathbf{p})$ be a pointed link in $S^3$ and $\bar{L}$ be the mirror of $L$. Pick a base point $p_0\in L\setminus \mathbf{p}$.
Then
there is a spectral sequence whose $E_2$-page is $\widetilde{Kh}'(\bar{L},\mathbf{p},p_0)$ and which converges to 
$\II^\natural(L,\mathbf{p}\cup\{p_0\})$.
\end{THE}
If $\mathbf{p}=\emptyset$, then the spectral sequences in Theorems \ref{ss2} and \ref{ss4} are
 Kronheimer-Mrowka's original ones in \cite{KM:Kh-unknot}
which relates $Kh(\bar{L})$ to $\II^\sharp(L)$ and $\widetilde{Kh}(\bar{L})$ to $\II^\natural(L)$ respectively. 

Suppose $L$ has $m$ components and
the base point $p_0$ lies in the last component of $L$ so that $\widetilde{Kh}(L,p_0)$ is equipped with a 
$R_{m-1}$ module structure. 
Given any marking point $p_i$, it determines a element $X_{p_i}\in R_m$ by the component of $L$ it lies on.
The point $p_i$ also determines an element in $R_{m-1}$ by requiring $X_{p_i}=0$ if $p_i$ lies on the $m$-th component. 
\begin{PR}\label{Koszul-Kh}
Suppose $(L,\mathbf{p}=\{p_i\}_{1\le i\le l})$ is a pointed link in $S^3$. 
Let 
$\mathbf{X}=(X_{p_1},\cdots, X_{p_l})\in R_m^{\oplus l}$ and  $\mathbf{X}'=(X_{p_1},\cdots, X_{p_l})\in R_{m-1}^{\oplus l}$ 
There is a spectral sequence whose $E_1$-page is the Koszul complex $K(\mathbf{X},Kh(L))$ and which converges to $Kh(L,\mathbf{p})$.
There is also 
a spectral sequence whose $E_1$-page is the Koszul complex $K(\mathbf{X}',\widetilde{Kh}(L,p_0))$ 
and which converges to $\widetilde{Kh}(L,\mathbf{p},p_0)$.
\end{PR}
\begin{proof}
Filter $CKh(L,\mathbf{p})$ and $\widetilde{CKh}(L,\mathbf{p},p_0)$ by the degree map
\begin{equation*}
\deg (y_{i_1}\wedge\cdots \wedge y_{i_k}\otimes b )= k
\end{equation*}
It is clear the associated spectral sequences are the ones in the proposition. 
\end{proof}

If we only resolve crossings in $M$ instead of resolving all the crossings in $N'$, any resolved link $L_v$ ($v\in \{0,1\}^M$)
is a copy of $L$. Again by \cite{KM:Kh-unknot}*{Corollary 6.8}, we have the following spectral sequence..
\begin{PR}\label{ss3}
Suppose $(L,\mathbf{p}=\{p_i\}_{1\le i \le l})$ is a pointed link in $S^3$, then
there is a spectral sequence converging to $\II^\sharp(L,\mathbf{p})$ whose $E_1$-page is 
\begin{equation}\label{cube-Lp-L}
\bigoplus_{v\in  \{0,1\}^{M}} \II^\sharp(L)
\end{equation}
which is a $l$-dimensional cube with a copy of $\II^\sharp(L)$ at each vertex
and the differential is the sum of
\begin{equation*}
d_{vz}=\pm \sigma_{p_i}
\end{equation*}
with a proper choice of sign where $vz$ is an edge of the cube such that $v(c_i)-z(c_i)=1$. 
\end{PR}
Given a point $p$ on $L$, the operator $\sigma_p$ on $\II^\sharp(L)$ only depends on the component of $L$ where it lies.
Using Proposition \ref{unknot-gen} and the functoriality of $\II^\sharp$, it is easy to see that $\sigma_p^2=0$.
Suppose $L$ has $m$ components, then using $\sigma_p$ we can equip $\II^\sharp(L)$ with a 
\begin{equation*}
R'_m=\mathbb{Z}[x_1,x_2\cdots,x_m]/(x_1^2,\cdots,x_m^2)
\end{equation*}
module structure. Using this module structure, the $E_1$-page in Proposition \ref{ss3} can be interpreted as a Koszul complex
of the $R'_m$-module $\II^\sharp(L)$. We can embed $R'_m$ into the ring $R_m$ defined in \eqref{Rm} by
\begin{equation}\label{x=2X}
x_j\to 2X_j
\end{equation}
So $Kh(L)$ is also equipped with a $R_m'$ module structure. If we are working with $\mathbb{Q}$-coefficients, the 2-factor in \eqref{x=2X}
does not cause any essential difference since $R_m'\otimes \mathbb{Q}=R_m\otimes \mathbb{Q}$.
Now both $Kh(\bar{L})$ and $\II^\sharp(L)$ are equipped with $R'_m$-module structures. 
We have the following result.
\begin{PR}
The spectral sequence relating $Kh(\bar{L})$ to $\II^\sharp(L)$ is a spectral sequence of $R'_m$-modules.
\end{PR}
The proof is exactly the same as the proof of \cite{AHI}*{Proposition 5.8}, so we skip the proof here. A similar result
in the Heegaard Floer homology case is proved in \cite{HN-module}.

\section{Earrings and sutures}
In this section we want to discuss the relationship between the (singular) instanton Floer homology for pointed links and 
the instanton Floer homology for sutured manifolds developed in \cite{KM:suture,KM-Alex}. 

Given a balanced sutured manifold $(M,\gamma)$, Kronheimer and Mrowka defined its instanton Floer homology
$\SHI(M)$ as the generalized eigenspace of certain operators on the instanton Floer homology of a closed three manifold \cite{KM:suture}.  
In particular, $\SHI(M)$ is defined over $\mathbb{C}$. Given a knot $K$ in a (closed oriented) three-manifold $Y$, 
its instanton knot Floer homology $\KHI(Y,K)$ is defined as $\SHI(M_K,\gamma_K)$ where $M_K$ is the knot complement $Y\setminus N(K)$
and $\gamma_K$ consists of two oppositely-oriented meridian sutures on $\partial M_K$. The definition is generalized for links
in \cite{KM-Alex} following the story in \cite{Juh-suture}: Given a link $L$ in $Y$, let $M_L$ be the link complement  
and $\gamma_L$ be pairs of oppositely-oriented meridian sutures on the boundary tori (exactly one pair on each boundary torus). 
Then define
\begin{equation*}
\KHI(Y,L):= \SHI(M_L,\gamma_L)
\end{equation*}
Given $(Y,L)$ as above, we pick marking points $\mathbf{p}=\{p_i\}_{1\le i \le l}$ with exactly \emph{one} point on each component of 
$L$. Even though we do not assume $Y=S^3$, we can still add earrings $(m_i,u_i)$ 
at $\mathbf{p}$ as in Section \ref{earring+triangle}. 
\begin{PR}\label{SHI=I}
Let $(Y,L,\mathbf{p})$ and the earring $(m_i,u_i)$ be given as above. Then we have
\begin{equation*}
\KHI(Y,L)\cong \II(Y,L\cup \bigcup_i m_i,\sum_i u_i;\mathbb{C})
\end{equation*}
\end{PR}
\begin{proof}
When $L$ is a knot, this is the content of \cite{KM:Kh-unknot}*{Proposition 1.4}. In this case, there is a single marking point $p\in L$ hence
a single earring $(m,u)$.
Cut $Y$ along the boundary tori $\partial N(L)$ and $\partial N(m)$ and re-glue using a diffeomorphism that maps the longitude of 
$L$ to the meridian of $m$, the resulting manifolds are a closed manifold $Y_L$ with a 1-cycle $u'$
and a triple $(S^3,H, w)$. The excision theorem
(\cite{KM:Kh-unknot}*{Theorem 5.6} or \cite{KM:suture}*{Theorem 7.7}) shows that $\II(Y_L,u';\mathbb{C})$ 
is isomorphic to $\II(Y,L\cup m,u;\mathbb{C})^{\oplus 2}$. On the other hand, the manifold $Y_L$ can also be obtained by attaching an annulus along
the two meridian sutures and then close up the manifold as described in \cite{KM:suture}*{Section 5.1}. It is exactly the closed manifold 
used to define $\KHI(Y,L)$ and $\II(Y_L,u';\mathbb{C})\cong \KHI(Y,L)\oplus \KHI(Y,L)$. Therefore 
$\KHI(Y,L)\cong \II(Y,L\cup m,u;\mathbb{C})$. 

For a general link $L=\cup_i L_i$ we can do the similar excision to all the pairs $(L_i,m_i)$
to obtain a pair $(Y_L,u')$. We have 
$$\II(Y_L,u';\mathbb{C})\cong\II(Y,L\cup\bigcup_i m_i,\sum_i u_i;\mathbb{C})^{\oplus 2}$$
The manifold $Y_L$ can also be obtained by
attaching an annulus to each pair of oppositely-oriented meridian sutures and closing up the manifold. According to
\cite{KM-Alex}*{Section 2.3}, this manifold can be used to define $\KHI(Y,L)$ and 
$$\II(Y_L,u';\mathbb{C})\cong \KHI(Y,L)\oplus \KHI(Y,L)$$
So we have
$$\II(Y,L\cup\bigcup_i m_i,\sum_i u_i;\mathbb{C})\cong \KHI(Y,L)$$
\end{proof}
Now we assume $(L,\mathbf{p})$ is non-degenerate.
A sutured manifold $(M_L,\gamma_{\mathbf{p}})$ can be defined 
by requiring that $M_L$ is the link complement and $\gamma_{\mathbf{p}}$ consists of a pair of oppositely-oriented meridian sutures
for each point in $\mathbf{p}$. 

We define another pointed link $(\widehat{L},\widehat{\mathbf{p}})$ with $\widehat{L}=L\cup U_1$ and 
$\widehat{\mathbf{p}}=\mathbf{p}\cup \{p_0\}$ where $p_0$ is a marking point on the unknot $U_1$. 
There is also an associated sutured manifold $(M_{\widehat{L}},\gamma_{\widehat{\mathbf{p}}})$ for $(\widehat{L},\widehat{\mathbf{p}})$.
The unreduced
Heegaard knot Floer homology for $(L,\mathbf{p})$ is defined as the sutured Heegaard Floer homology for 
$(M_{\widehat{L}},\gamma_{\widehat{\mathbf{p}}})$ in \cite{BLS}. We can  take the instanton Floer homology for this sutured manifold
to obtain an invariant for $(L,\mathbf{p})$. Pick a subset $\mathbf{p}'\subset {\mathbf{p}}$ with exactly one
point on each component of ${L}$ and $\widehat{\mathbf{p}}'=\mathbf{p}'\cup\{p_0\}$. 
Now we have $\gamma_{{\mathbf{p}}}$ and $\gamma_{{\mathbf{p}}'}$
($\gamma_{\widehat{\mathbf{p}}}$ and $\gamma_{\widehat{\mathbf{p}}'}$) differ by
$(|{\mathbf{p}}|-|\mathbf{p}'|)$ pairs of meridian sutures. According to \cite{KM-Alex}*{Section 3.1}, whenever we add a pair of
meridian sutures, the dimension of the instanton Floer homology is multiplied by $2$. So we have
\begin{equation}\label{SHI-p-p'-non-reduced}
\SHI(M_{{L}},\gamma_{{\mathbf{p}}})\cong \SHI(M_{{L}},\gamma_{\mathbf{p}'})\otimes_{\mathbb{C}} 
\mathbb{C}^{2^{(|{\mathbf{p}}|-|\mathbf{p}'|) }}
\end{equation}
\begin{equation}\label{SHI-p-p'}
\SHI(M_{\widehat{L}},\gamma_{\widehat{\mathbf{p}}})\cong \SHI(M_{\widehat{L}},\gamma_{\widehat{\mathbf{p}}'})\otimes_{\mathbb{C}} 
\mathbb{C}^{2^{(|{\mathbf{p}}|-|\mathbf{p}'|) }}
\end{equation}

\begin{PR}\label{SHI=I-non-degenerate}
Suppose $(L,\mathbf{p})$ is a non-degenerate pointed link in $S^3$ and 
$(M_{{L}},\gamma_{{\mathbf{p}}})$,
$(M_{\widehat{L}},\gamma_{\widehat{\mathbf{p}}})$  are the sutured manifolds defined as above.
Then we have
\begin{equation*}
\SHI(M_{{L}},\gamma_{{\mathbf{p}}})\cong \II^\natural(L,\mathbf{p};\mathbb{C}),~~
\SHI(M_{\widehat{L}},\gamma_{\widehat{\mathbf{p}}})\cong \II^\sharp(L,\mathbf{p};\mathbb{C})
\end{equation*}
\end{PR}
\begin{proof}
Suppose $L$ has $m$ components. 
When $|\mathbf{p}|=m$, there is exactly one marking point on each component of $L$ since $(L,\mathbf{p})$ is non-degenerate. 
In this situation, the proposition follows from Proposition \ref{SHI=I}.
The general case follows from Proposition \ref{p+q=2I-Z}, \eqref{SHI-p-p'-non-reduced}, \eqref{SHI-p-p'} and induction.
\end{proof}

\begin{LE}\label{Isharp=2I}
Suppose $(L,\mathbf{p})$ is a pointed link in $S^3$ with non-empty $\mathbf{p}$. Then we have
\begin{equation*}
\dim \II^\sharp(L,\mathbf{p};\mathbb{F})=2\dim \II^\natural(L,\mathbf{p};\mathbb{F})
\end{equation*}
where $\mathbb{F}$ is any field of characteristic $p\neq 2$.
\end{LE}
\begin{proof}
Let $Y_1$ and $Y_2$ be two homology 3-spheres.
Fukaya's connected sum theorem \cite{Fukaya-sum} relates the instanton Floer homologies 
$\II(Y_1;\mathbb{Z}),\II(Y_2;\mathbb{Z})$ to $I(Y_1\# Y_2;\mathbb{Z})$ by a 
spectral sequence. An elaboration of Fukaya's theorem with $\mathbb{Q}$-coefficients
can be found in \cite{Don:YM-Floer}*{Section 7.4}. 
Indeed Donaldson's proof works for any field $\mathbb{F}$ whose characteristic is not $2$.
The same problem is 
also studied in \cite{Li-sum}. We want to apply Fukaya's theorem to the connected sum 
\begin{equation*}
(S^3, H\cup L\cup\bigcup_i m_i, w+\sum_i u_i )=(S^3,H, w)\# (S^3,L\cup \bigcup_i m_i, \sum_i u_i)
\end{equation*} 
Even though this is not
a connected sum of homology 3-spheres with empty links in Fukaya's setting, the same proof still works. Our situation is
even easier because we are working in an admissible case so that reducible critical points of the Chern-Simons functional 
do not enter into the discussion.  

From \cite{Don:YM-Floer}*{Section 7.4}, we have a spectral sequence which converges to $\II^\sharp(L,\mathbf{p};\mathbb{F})$ and
whose last possibly non-degenerate page is 
\begin{align*}
\xymatrix{
  \II^\natural(L,\mathbf{p};\mathbb{F})\otimes I(S^3,  H,w;\mathbb{F}) \ar[d]^{f}  \\
  \II^\natural(L, \mathbf{p};\mathbb{F})\otimes I(S^3,  H,w;\mathbb{F})   }
\end{align*}
where the differential $f$ is $2\mu(x_1)\otimes 1- 1\otimes 2\mu (x_2)$ ($x_1\in S^3\setminus L$ and $x_2\in S^3\setminus H$). 
We use
$1$ for the identify map and
$\mu(x)$ for the point operator of degree 4. Our convention for $\mu$ is from \cite{DK}. 
Given any admissible triple $(Y,K,u)$,
\cite{Kr-ob}*{Formula (47)} says that 
\begin{equation*}
\sigma_p^2=8-4\mu(x)
\end{equation*}
for any $p\in K$ and $x\in Y\setminus K$. Pick a point $p$ on an earing $m_i$ for $L$, the orbifold bundle does not extend to $m_i$
and the discussion in Section \ref{s-inst}
shows that $2\sigma_p=0$. Since we are working over field $\mathbb{F}$, we have $\sigma_p=0$ and $\mu(x_1)=2$. Similarly we have
$\mu(x_2)=2$. Therefore the differential $f=0$ and the proposition follows from the spectral sequence. 
\end{proof}
\begin{rmk}
The above lemma does \emph{not} hold when $\mathbb{F}=\mathbb{Z}/2$. Indeed we have
\begin{equation*}
\dim \II^\sharp(L,\mathbf{p};\mathbb{Z}/2)=4\dim \II^\natural(L,\mathbf{p};\mathbb{Z}/2)
\end{equation*}
which can be proved using Proposition \ref{I-Lp-Z2} and \cite{KM-jsharp}*{Lemma 7.7}.
\end{rmk}

\begin{THE}\label{final-inequality}
Suppose $L$ be a link of $m$ components in $S^3$ and $p_0$ is a base point for $L$ lying on the $m$-th component. 
Let $\mathbf{X}=(X_1,\cdots, X_m)\in R_m^{\oplus m}$ and $\mathbf{X}'=(X_1,\cdots, X_{m-1})\in R_{m-1}^{\oplus m-1}$. Then we have
\begin{equation*}
2\dim_{\mathbb{C}} \KHI(L)\le \rank_{\mathbb{Z}}H(K(\mathbf{X}, Kh(L))) 
\end{equation*}
and 
\begin{equation*}
\dim_{\mathbb{C}} \KHI(L)\le \rank_{\mathbb{Z}}H(K(\mathbf{X}', \widetilde{Kh}(L,p_0))) 
\end{equation*}
\end{THE}
\begin{proof}
Pick marking points $\{p_1,\cdots,p_m\}$ where $p_i$ lies on the $i$-th component of $L$. 
By Proposition \ref{SHI=I}, we have
\begin{equation*}
\KHI(L)\cong \II^\natural(L,\{p_1,\cdots,p_m\};\mathbb{C})
\end{equation*}
By Lemma \ref{Isharp=2I}, we have
\begin{equation*}
   2\dim \II^\natural(L,\{p_1,\cdots,p_m\};\mathbb{C}) = \dim\II^\sharp(L,\{p_1,\cdots,p_m\};\mathbb{C})
\end{equation*}
The theorem follows from Theorem \ref{ss2}, Theorem \ref{ss4}, Proposition \ref{Koszul-Kh} and the facts that
$Kh(\bar{L};\mathbb{Q})=Kh({L};\mathbb{Q})^\ast$ and 
$\widetilde{Kh}(\bar{L};\mathbb{Q})=\widetilde{Kh}({L};\mathbb{Q})^\ast$.
\end{proof}

\bibliography{references}

\begin{bibdiv}
\begin{biblist}

\bib{DB:sur-rel}{incollection}{
      author={Braam, Peter},
      author={Donaldson, Simon},
       title={Floer's work on instanton homology, knots and surgery},
        date={1995},
   booktitle={The {F}loer memorial volume},
      series={Progr. Math.},
      volume={133},
   publisher={Birkh\"auser},
     address={Basel},
       pages={195\ndash 256},
      review={\MR{1362829 (97i:57034)}},
}

\bib{BL}{article}{
      author={Baldwin, John~A.},
      author={Levine, Adam~Simon},
       title={A combinatorial spanning tree model for knot {F}loer homology},
        date={2012},
        ISSN={0001-8708},
     journal={Adv. Math.},
      volume={231},
      number={3-4},
       pages={1886\ndash 1939},
         url={https://doi.org/10.1016/j.aim.2012.06.006},
      review={\MR{2964628}},
}

\bib{BLS}{article}{
      author={Baldwin, John~A.},
      author={Levine, Adam~Simon},
      author={Sarkar, Sucharit},
       title={Khovanov homology and knot {F}loer homology for pointed links},
        date={2017},
        ISSN={0218-2165},
     journal={J. Knot Theory Ramifications},
      volume={26},
      number={2},
       pages={1740004, 49},
         url={https://doi.org/10.1142/S0218216517400041},
      review={\MR{3604486}},
}

\bib{BS}{article}{
      author={Baldwin, John~A},
      author={Sivek, Steven},
       title={Khovanov homology detects the trefoils},
        date={2018},
     journal={arXiv preprint arXiv:1801.07634},
}

\bib{DK}{book}{
      author={Donaldson, Simon},
      author={Kronheimer, Peter},
       title={The geometry of four-manifolds},
   publisher={Oxford University Press},
        date={1990},
}

\bib{Don:YM-Floer}{book}{
      author={Donaldson, Simon},
       title={Floer homology groups in {Y}ang-{M}ills theory},
      series={Cambridge Tracts in Mathematics},
   publisher={Cambridge University Press, Cambridge},
        date={2002},
      volume={147},
        ISBN={0-521-80803-0},
         url={http://dx.doi.org/10.1017/CBO9780511543098},
        note={With the assistance of M. Furuta and D. Kotschick},
      review={\MR{1883043 (2002k:57078)}},
}

\bib{Fukaya-sum}{article}{
      author={Fukaya, Kenji},
       title={Floer homology of connected sum of homology {$3$}-spheres},
        date={1996},
        ISSN={0040-9383},
     journal={Topology},
      volume={35},
      number={1},
       pages={89\ndash 136},
         url={https://doi.org/10.1016/0040-9383(95)00009-7},
      review={\MR{1367277}},
}

\bib{HN-module}{article}{
      author={Hedden, Matthew},
      author={Ni, Yi},
       title={Khovanov module and the detection of unlinks},
        date={2013},
        ISSN={1465-3060},
     journal={Geom. Topol.},
      volume={17},
      number={5},
       pages={3027\ndash 3076},
         url={https://doi.org/10.2140/gt.2013.17.3027},
      review={\MR{3190305}},
}

\bib{Juh-suture}{article}{
      author={Juh\'{a}sz, Andr\'{a}s},
       title={Holomorphic discs and sutured manifolds},
        date={2006},
        ISSN={1472-2747},
     journal={Algebr. Geom. Topol.},
      volume={6},
       pages={1429\ndash 1457},
         url={https://doi.org/10.2140/agt.2006.6.1429},
      review={\MR{2253454}},
}

\bib{Kh-module}{article}{
      author={Khovanov, Mikhail},
       title={Patterns in knot cohomology. {I}},
        date={2003},
        ISSN={1058-6458},
     journal={Experiment. Math.},
      volume={12},
      number={3},
       pages={365\ndash 374},
         url={http://projecteuclid.org/euclid.em/1087329238},
      review={\MR{2034399}},
}

\bib{KM-Alex}{article}{
      author={Kronheimer, P.~B.},
      author={Mrowka, T.~S.},
       title={Instanton {F}loer homology and the {A}lexander polynomial},
        date={2010},
        ISSN={1472-2747},
     journal={Algebr. Geom. Topol.},
      volume={10},
      number={3},
       pages={1715\ndash 1738},
         url={https://doi.org/10.2140/agt.2010.10.1715},
      review={\MR{2683750}},
}

\bib{KM:suture}{article}{
      author={Kronheimer, P.~B.},
      author={Mrowka, T.~S.},
       title={Knots, sutures, and excision},
        date={2010},
     journal={Journal of Differential Geometry},
      volume={84},
      number={2},
       pages={301\ndash 364},
}

\bib{KM:Kh-unknot}{article}{
      author={Kronheimer, P.~B.},
      author={Mrowka, T.~S.},
       title={Khovanov homology is an unknot-detector},
        date={2011},
        ISSN={0073-8301},
     journal={Publ. Math. Inst. Hautes \'Etudes Sci.},
      number={113},
       pages={97\ndash 208},
         url={https://doi.org/10.1007/s10240-010-0030-y},
      review={\MR{2805599}},
}

\bib{KM-filtration}{article}{
      author={Kronheimer, P.~B.},
      author={Mrowka, T.~S.},
       title={Filtrations on instanton homology},
        date={2014},
        ISSN={1663-487X},
     journal={Quantum Topol.},
      volume={5},
      number={1},
       pages={61\ndash 97},
         url={https://doi.org/10.4171/QT/47},
      review={\MR{3176310}},
}

\bib{KM-jsharp}{article}{
      author={Kronheimer, P.~B.},
      author={Mrowka, T.~S.},
       title={Tait colorings, and an instanton homology for webs and foams},
        date={2015},
     journal={arXiv preprint, arXiv:1508.07205},
}

\bib{KM-surface1}{article}{
      author={Kronheimer, P.~B.},
      author={Mrowka, T.~S.},
       title={Gauge theory for embedded surfaces. {I}},
        date={1993},
        ISSN={0040-9383},
     journal={Topology},
      volume={32},
      number={4},
       pages={773\ndash 826},
         url={https://doi.org/10.1016/0040-9383(93)90051-V},
      review={\MR{1241873}},
}

\bib{Kr-ob}{article}{
      author={Kronheimer, Peter~B.},
       title={An obstruction to removing intersection points in immersed
  surfaces},
        date={1997},
     journal={Topology},
      volume={36},
      number={4},
       pages={931\ndash 962},
}

\bib{Li-sum}{article}{
      author={Li, Weiping},
       title={Floer homology for connected sums of homology {$3$}-spheres},
        date={1994},
        ISSN={0022-040X},
     journal={J. Differential Geom.},
      volume={40},
      number={1},
       pages={129\ndash 154},
         url={http://projecteuclid.org/euclid.jdg/1214455289},
      review={\MR{1285531}},
}

\bib{OS-ss}{article}{
      author={Ozsv\'ath, Peter},
      author={Szab\'o, Zolt\'an},
       title={On the {H}eegaard {F}loer homology of branched double-covers},
        date={2005},
        ISSN={0001-8708},
     journal={Adv. Math.},
      volume={194},
      number={1},
       pages={1\ndash 33},
         url={https://doi.org/10.1016/j.aim.2004.05.008},
      review={\MR{2141852}},
}

\bib{AHI}{article}{
      author={Xie, Yi},
       title={Instantons and annular {K}hovanov homology},
        date={2018},
     journal={arXiv preprint, arXiv:1809.01568},
}

\end{biblist}
\end{bibdiv}
\bibliographystyle{hplain}

\Address

\end{document}